\documentclass[graybox]{svmult}

\usepackage{mathptmx}       
\usepackage{helvet}         
\usepackage{courier}        
\usepackage{makeidx}         
\usepackage{graphicx}        
\usepackage{multicol}        
\usepackage[bottom]{footmisc}
\usepackage{amsmath,amssymb,amsfonts}        
\makeindex             

\begin{document}
\title*{Morse structures on partial open books with extendable monodromy}
\author{Joan E. Licata and Daniel V. Mathews}
\institute{Joan E. Licata \at Mathematical Sciences Institute, The Australian National University, \email{joan.licata@anu.edu.au}
\and Daniel V. Mathews \at School of Mathematical Sciences, Monash University \email{Daniel.Mathews@monash.edu}}
%
%
\maketitle

\abstract{The first author in recent work with D. Gay developed the notion of a \emph{Morse structure} on an open book as a tool for studying closed contact 3-manifolds. We extend the notion of Morse structure to \emph{extendable partial} open books in order to study contact 3-manifolds with convex boundary.
}

\def\jel#1{{\textcolor{red}{#1}}}
\def\dm#1{{\textcolor{blue}{#1}}}


\section{Introduction}

In \cite{Gay_Licata15}, the first author and David Gay developed the notion of a \emph{Morse structure} on a closed 3-manifold with an open book decomposition. Informally, a Morse structure is a nice family of functions and vector fields on the pages of the open book: the functions are Morse functions on the pages, and the vector fields are gradient-like and Liouville in an appropriate sense. In \cite{Gay_Licata15} it was shown that every open book admits a Morse structure. 

The same paper \cite{Gay_Licata15} also developed the notion of a \emph{Morse diagram}. This is a diagram consisting of some tori, one for each binding component, with some curves and decorations drawn on them. A Morse structure on an open book has a Morse diagram, and \cite{Gay_Licata15} (prop. 3.7) showed that every abstract Morse diagram arises as the Morse diagram of an open book. This gives a graphical description, encoded by a finite amount of combinatorial data, of an open book and hence of a contact structure.

Morse structures and diagrams give a useful way to study \emph{Legendrian knots and links} in a closed contact 3-manifold. A Legendrian knot or link in the standard contact $\mathbb{R}^3$ can be studied via its \emph{front projection}, which projects the knot into a plane, and whose distance from the plane at any point is determined by the slope of the projection. In an analogous way, a Morse structure allows one to define a front projection for (almost) any Legendrian knot or link in any contact manifold. By flowing the link to a neighbourhood of the binding, one obtains a \emph{front} for the link on the associated Morse diagram, and the slope of the diagram at any point determines the ``distance" of the link from the binding. Fronts were defined in \cite{Gay_Licata15}, along with a set of ``Reidemeister moves": two Legendrian links represented by fronts are Legendrian isotopic if and only if their fronts are related by such moves.

The purpose of this short article is to explore a simple idea: what happens if we look at \emph{partial} open books defined by \emph{restrictions} of the monodromies in the closed case? We examine the consequences of \cite{Gay_Licata15} in this context, and extend the results to a large family of contact 3-manifolds with convex boundary. We generalise \cite{Gay_Licata15} to partial open books whose monodromy is \emph{extendable} to the monodromy of an open book in the usual (non-relative) sense.

\emph{Partial open books} were introduced by Honda--Kazez--Mati\'{c} in \cite{HKM09}. They are related to open books in the same way that contact 3-manifolds with convex boundary are related to closed contact 3-manifolds. In \cite{HKM09} Honda--Kazez--Mati\'{c} stated a relative version of the Giroux Correspondence between contact manifolds and open books \cite{Gi02}, which was also expounded by Etg\"{u}--Ozbagci in \cite{Etgu_Ozbagci11}.

Following \cite{Gay_Licata15}, define  a contact manifold $W$ with a contact form $\alpha$ by
\[
W = (0, \infty) \times S^1 \times S^1, \quad
\alpha = dz + x \; dy,
\]
where $x,y,z$ are coordinates on the three factors of $W$. We prove the following.

\begin{theorem}\label{thm:intro}
Let $(M, \Gamma, \xi)$ be a contact 3-manifold with convex boundary, presented by the partial open book $(S,P,h)$, with binding $B$. There is a 2-complex $\text{Skel} \subset \text{Int} M$ with the property that, after modifying $\xi$ by an isotopy through contact structures presented by $(S,P,h)$, the interior of each connected component of $(M \backslash \left( \text{Skel} \cup B  \right), \xi)$ is contactomorphic to a contact submanifold of $W$.
\end{theorem}

Once sufficient notation has been established, in Section~\ref{sec:contacto} we  give a more precise description of these submanifolds in terms of the defining data $(S,P,h)$ of an abstract open book defining $(M, \Gamma, \xi)$.  

In Section~\ref{sec:diagram} we define a Morse structure for an extendable partial open book $(S,P,h)$.  A Morse structure consists of a function $F$ and a vector field $V$, and this data can be used to define a \textit{Morse diagram}, which is a decorated surface consisting of tori, punctured tori and annuli.  A Morse diagram can be viewed as gluing instructions for assembling $\text{Skel}$ and submanifolds of $W$ into the original manifold $M$. The components of the Morse diagram are properly  embedded in $M$ and transverse to the vector field $V$ along the pages of the partial open book. The flow of $V$ assigns to points in the complement of $\text{Skel}$  and the binding a well-defined image on the Morse diagram, which we call a \textit{front}.

\begin{theorem}
\label{thm:Legendrian_tangles}  
If $\Lambda$ is a properly embedded Legendrian tangle in $(M, \xi)$ disjoint from the binding $B$ and transverse to $\text{Skel}$, then the front associated to $\Lambda\setminus \text{Skel}$ completely determines  $\Lambda$.  Consequently, any two Legendrian tangles with the same front are equal.

\end{theorem}

Fronts can effectively distinguish Legendrian tangles up to Legendrian isotopy.

\begin{theorem}
\label{thm:Legendrian_moves} The set of moves shown in Figure~\ref{fig:Legendrian_moves} has the property that two Legendrian tangles in $(M, \Gamma, \xi)$ are Legendrian isotopic if and only if their fronts are related by a sequence of moves and by isotopy preserving sufficiently negative slope.
\end{theorem}

We illustrate the ideas with an example adapted from  \cite{Etgu_Ozbagci11}; see figure~\ref{fig:eoex}. The right hand figures show $P\subset S$ and $h(P)\subset S$. The gluing map $h$ extends to a homeomorphism of $S$ which is a given by a Dehn twist around a curve parallel to the exterior boundary component.  The three boundary components  of $S$ each correspond to a component of the Morse diagram shown on the left, and the thin curves encode the extended monodromy.  The bold curve on the Morse diagram is a front projection of a Legendrian tangle with one closed component and one properly embedded interval component.

\begin{figure}
\begin{center}
\includegraphics[scale=0.6]{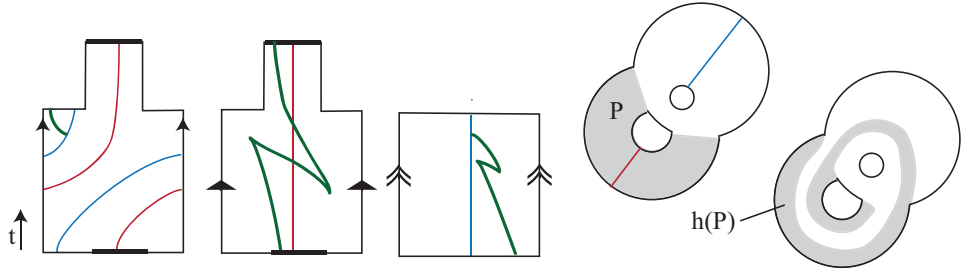}
\caption{Morse diagram for the extendable open book $(S,P, h)$, shown with a front for a Legendrian tangle (bold).  The bold segments at the top and bottom are identified, as are vertical edges as indicated by arrows.}
\label{fig:eoex}
\end{center}
\end{figure}

We conclude this section with a brief remark about gluing.  Contact manifolds   may be glued along compatible convex boundaries, and the simplest case of this is gluing contact manifolds which are products.  This gluing can be represented on the Morse diagram level by stacking Morse diagrams. Front projection of Legendrian tangles also behaves nicely under this operation.  In the special case of tangles braided with respect to the product structure, front projection offers a new tool for studying Legendrian braids in product manifolds.

\begin{acknowledgement} 
The authors would like to acknowledge the support and hospitality of MATRIX during the workshop  \textit{Quantum Invariants and Low-Dimensional Topology}.
The second author is supported by Australian Research Council grant DP160103085.

\end{acknowledgement}

\section{Partial open books}

We follow the definition of partial open books in \cite{Etgu_Ozbagci11}. 
All handles will be assumed two-dimensional, so a \emph{0-handle} is  a closed disc $D^2$ and a \emph{1-handle} is a closed oriented 2-disc of the form $P_0 = [-1,1] \times [-1,1]$.  To add a $1$-handle  to an oriented surface $S$, select an embedded $0$-sphere $\{p,q\} \in \partial S$ called the \textit{attaching sphere} and identify a regular neighbourhood of $p,q$ with  $[-1,1] \times \{-1,1\} \subset P_0$ in an orientation-preserving fashion. Any connected oriented surface with nonempty boundary can be constructed by successively attaching 1-handles to 0-handles. The \emph{core} of a handle is $\{0\} \times [-1,1]$ and the \emph{co-core} is $[-1,1] \times \{0\}$. We note that a handle attachment may be undone by cutting an attached handle through its co-core and deformation retracting it onto its attaching intervals.

Throughout this paper, $(S,P)$ denotes a pair of compact oriented surfaces, with $P \subseteq S$, $S$ connected and $\partial S \neq \emptyset$. We allow $P = \emptyset$ and $P = S$.

\begin{definition}
\label{def:handle_structure}
A \emph{handle structure} compatible with $(S,P)$ is a sequence of 1-handles $P_1, P_2, \ldots, P_r$ in $S$ such that $P = P_1 \cup \cdots \cup P_r$ and $S$ is obtained from $\overline{S \backslash P}$ by successively attaching 1-handles $P_1, \ldots, P_r$.
\end{definition}
When we have such a handle structure, for convenience we write $R = \overline{S \backslash P}$. Thus $S$ is obtained form $R$ by attaching the 1-handles of $P$. Note then that each component of $\partial P$ is either a component of $\partial S$ or a concatenation of arcs alternating between  $\partial P \cap \partial S$ and $\partial P \backslash \partial S$. We will denote $A = \partial P \cap \partial S$.

\begin{definition}
An \emph{abstract partial open book} is a triple $(S,P,h)$ where $(S,P)$ admits a compatible handle structure and $h : P \rightarrow S$ is a homeomorphism onto its image such that $h$ is the identity on $A$.
\end{definition}
The function $h$ is called the \emph{monodromy}. Note when $P=\emptyset$, $h$ is the null function. When  $P=S$, $h$ is a homeomorphism of $S$ to itself fixing the boundary, and we obtain an (abstract) open book in the usual sense.

This definition of abstract partial open book differs slightly from Honda--Kazez--Mati\'{c} in \cite{HKM09}, who consider pairs $(S,P)$ where $P$ is a subsurface of $S$ such that each component of $\partial P$ is either contained in $\partial S$ or is polygonal with every second side in $\partial S$. As noted above, any $(S,P)$ admitting a compatible handle structure has this form, but the \cite{HKM09} definition also allows bigon components of $P$ with one side in $A$. Such a \emph{boundary-parallel bigon} deformation retracts into $A$ and one can show that the resulting contact manifold is contactomorphic to the original one. In effect, then, the definitions are equivalent.

Clearly the existence of a compatible handle structure on $(S,P)$ restricts  the topology of $S$ and $P$. For the reasons discussed above,  no component of $\partial P$ can lie in $\text{Int} S$, and no component of $P$ is a boundary-parallel bigon.

Following  \cite{Etgu_Ozbagci11}, from a partial open book decomposition $(S,P,h)$ we  construct a sutured 3-manifold as follows. We define two handlebodies by thickening $S$ and $P$ and collapsing portions of their boundaries:  
\[
H = \frac{S \times [-1,0]}{\text{$(x,t) \sim (x,t')$ for $x \in \partial S$ and $t \in [-1,0]$}}
\]
\[
N = \frac{P \times [0,1]}{\text{$(x,t) \sim (x,t')$ for $x \in A$ and $t \in [0,1]$}}.
\]
(Note we only collapse the part of the boundary along $A = \partial P \cap \partial S$, leaving $(\partial P \backslash \partial S) \times [0,1]$ unscathed.) Now glue these two handlebodies together,  along both the common $P \times \{0\} \subseteq S \times \{0\}$ and also by identifying points $(x,1) \sim (h(x), -1)$ for $x \in P$.

The resulting manifold is denoted $M(S,P,h)$. It has boundary given by
\[
R \times \{0\} \cup \overline{(-S \backslash h(P))} \times \{-1\} \cup (\partial P \backslash \partial S) \times [0,1]
\]
and \emph{binding given by $B = S \times \{0\}$,}
modulo the identifications above, and thus has a sutured structure, with sutures $\Gamma$ and complementary regions $R_\pm$ given by
\begin{gather*}
\Gamma = \overline{\partial P \backslash \partial S} \times \{1/2\} \cup \overline{\partial S \backslash \partial P} \times \{-1/2\}, \\
R_+ = R \times \{0\} = \overline{S \backslash P} \times \{0\}, \quad
R_- = \overline{-S \backslash h(P)} \times \{-1\}.
\end{gather*}
Since $h$ is a homeomorphism onto its image, $\chi(R_+) = \chi(R_-)$, so $M(S,P,h)$ is a \emph{balanced} sutured manifold in the sense of \cite{Ju06}. The sutured structure on the boundary of $(M, \Gamma)$ is equivalent to the structure of a dividing set for a convex surface in a contact manifold \cite{Gi91}.

Indeed, to a partial open book $(S,P, h)$  we associate a contact manifold with convex boundary (up to contactomorphism), given by $M(S,P,h)$, with the unique (isotopy class of) contact structure whose restrictions to $H$ and $N$ are both tight, with dividing sets $\partial S \times \{-1/2\}$ and $\partial P \times \{1/2\}$ respectively \cite{Etgu_Ozbagci11, Torisu00}. Thus we regard $M(S,P,h)$ as a contact manifold. 
 
Following \cite{Etgu_Ozbagci11}, two partial open books $(S,P, h)$ and $(\overline{S}, \overline{P}, \overline{h})$ are said to be \textit{isomorphic} if there is a diffeomorphism $g:S\rightarrow \overline{S}$ such that $g(P)=\overline{P}$ and $\overline{h}=g\circ h\circ (g^{-1})|_{\overline{P}}$.
The \emph{relative Giroux Correspondence} establishes a bijection between isomorphism classes of partial open book decompositions, up to positive stabilisation, and compact contact 3-manifolds with convex boundary, up to contactomorphism \cite{Etgu_Ozbagci11, Gi00, HKM09}.

In order to generalise the notion of a Morse structure from a closed contact manifold to one with convex boundary, it is helpful to discuss particular manifolds rather than isomorphism classes, so we make the following definitions.

\begin{definition}
A closed contact manifold $(M, \xi)$ is \textit{presented by the  open book} $(S,h)$  if it is contactomorphic to $M(S, h)$.  A  contact manifold with convex boundary $(M, \Gamma, \xi)$ is \textit{presented by the partial open book}  $(S,P,h)$ if it is contactomorphic to  $M(S,P, h)$.
\end{definition}

In the remainder of this paper we will consider manifolds of the form $M(S,h)$ or $M(S, P, h)$ so all results are up to diffeomorphism.  In the case that the initial object is a manifold with an honest ---as opposed to abstract--- open book, the identifying diffeomorphism may be used to transfer structures from $M(S,h)$ or $M(S, P, h)$ to the given contact manifold.

\section{Slices}\label{sec:slice}

Up to isotopy, the pair $(S,P)$ may be encoded via a simple combinatorial diagram generated by the handle structure, which we call a \emph{slice} and define presently.

The first step in defining a slice is to extend the core and co-core of each handle to a $1$-complex.
Consider a compact connected oriented surface $S$ constructed from a finite collection of $0$-handles by successively attaching $1$-handles $P_1, P_2, \ldots, P_r$. Since we only consider handle structures up to isotopy, we are free to assume that the attaching spheres are disjoint from the corners where two handles meet and from the endpoints of any co-core.  When a point $p$ of the attaching sphere lies on the boundary of a $0$-handle, extend the core of $P_i$ through $p$ via a ray to the centre of the $0$-handle.  Now  assume that the cores of previous handles have already been extended.  When $p$ lies on the boundary of a $1$-handle,  there is a unique (up to isotopy) way to extend the core of $P_i$ through $1$-handles until it reaches a point on the boundary of a $0$ handle and satisfies the condition that  co-core of $P_j$ intersects the core of $P_k$ in $\delta_{jk}$ points for all $j,k\leq i$.  Then one may extend radially, as above.   We call the union of the co-cores and the extended cores the \textit{core complex} associated to the handle structure.  Note that $S$ deformation retracts onto its core complex. If, at each stage, we allow attaching points
to slide along the boundary, by isotopy in the complement of the co-cores, this core complex  is still determined up to isotopy.

Now consider a pair $(S,P)$ with a compatible handle structure as in definition \ref{def:handle_structure}. Then $S$ can be constructed from $0$-handles $D_1, \ldots, D_d$ by first adding $1$-handles $R_1, \ldots, R_r$ to form $R$ and then adding further $1$-handles $P_1, \ldots, P_p$  to form $S$. That is,
\[
R = D_1 \cup \cdots \cup D_d \cup R_1 \cup \cdots \cup R_r, \quad
P = P_1 \cup \cdots \cup P_p, \quad
S = R \cup P.
\]
In the corresponding core complex, each core and co-core arises from an $R_i$ or $P_j$.

The boundary $\partial S$ consists of finitely many circles, each of which inherits a boundary orientation from $S$. These circles contain the endpoints of all co-cores, which form $r+p$ pairs of points. Each circle either lies in $\partial S \backslash \partial P$, or in $A$, or decomposes into arcs alternately in $A$ and $\partial S \backslash \partial P$. We represent the arcs of $\partial S \backslash \partial P$ by an additional decoration --- a marker denoted by an X.
\begin{definition}
Let $r,p,q \geq 0$ be integers. A \emph{slice} $\mathcal{SL}$ is a collection of oriented circles, together with a set of decorations at $2(r+p)+q$ distinct points as follows:
\begin{enumerate}
\item
$r$ pairs of points called \emph{antecedent pairs}
\item
$p$ pairs of points called \emph{primary pairs}
\item
$q$ further points called \emph{markers}.
\end{enumerate}
The \emph{slice} of a handle structure $R_1, \ldots, 
R_r$, $P_1, \ldots, P_p$ on $(S,P)$ consists of $\partial S$, together with antecedent  
pairs given by endpoints of co-cores of the $R_i$, primary  
pairs given by endpoints of co-cores of the $P_j$, and a marker in each arc of $\partial S \backslash \partial P$.

\end{definition}

Figure \ref{fig:slices} shows two examples of pairs $(S,P)$ with handle structures, together with their core complexes and slices.

The oriented circles and pairs of points (antecedent and primary taken together) of a slice are sufficient to recover $S$, up to homeomorphism. To recover the pair $(S,P)$, however, we need the distinction between antecedent and primary pairs as well as   the markers.

\begin{remark} Slices bear a resemblance to the \emph{arc diagrams} of bordered Floer theory \cite{LOT08}, especially in the \emph{bordered sutured} case of \cite{Zarev09} or in the context of the \emph{quadrangulated surfaces} studied by the second author in \cite{Me16_strand}. This is not surprising, since both are essentially boundary data of handle decompositions of a surface, though slices have slightly more decoration.
\end{remark}

\begin{figure}
\begin{center}
\scalebox{.7}{\includegraphics{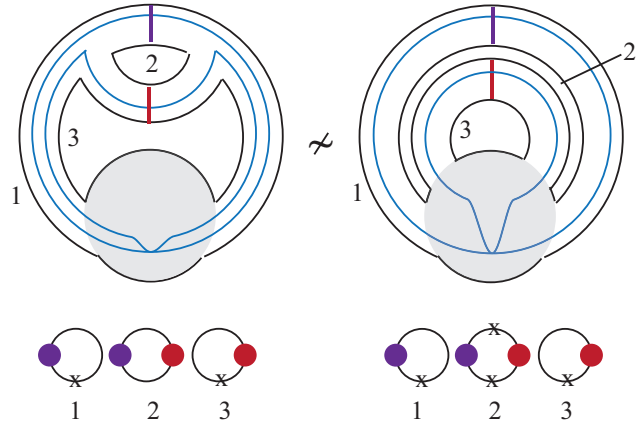}}
\caption{Two pairs $(S,P)$, together with core complexes and slices. In both figures, $P$ is white and $R$ is shaded. In both figures, $S$ is an annulus, and $R$ is a disc. However on the left $P$ is an annulus, while on the right $P$ consists of 2 discs. The two handle structures are related by an isotopy of attaching points which passes through arcs of $\partial S \backslash \partial P$, resulting in distinct slices.}
\label{fig:slices}
\end{center}
\end{figure}

\begin{lemma} 
If two pairs $(S,P)$, $(S',P')$ have isomorphic slices, then there is a homeomorphism of pairs $(S,P) \cong (S',P')$.
\end{lemma}

The proof explicitly reconstructs a surface pair from a slice.

\begin{proof} 

First consider the slice $\mathcal{SL}$ of a pair $(S,P)$. Surgery on $\mathcal{SL}$ at each pair of marked points (antecedent and primary) yields a 1-manifold which is the boundary of the surface formed by cutting  all the $1$-handles along their co-cores. This surgered surface is homeomorphic to the 0-handles, hence the number of components of the 1-manifold obtained by surgery on $\mathcal{SL}$ is equal to the number 
 of $0$-handles. In fact, the boundary of this surface naturally contains the markers, as well as the attaching spheres needed  to recover $R$ and $S$ in turn. We note that after reattaching the antecedent handles, the boundary contains primary pairs of points and markers, and each successive primary handle is attached at  points on the boundary of $R$ or on already-attached primary handles. 
Up to homeomorphism preserving $R$ and $P$ at each stage, there is no choice where to attach handles, so it follows that the slice determines the pair $(S,P)$.
\end{proof}

\begin{remark}The handle structures which appear in   \cite{Gay_Licata15} were required to have a  unique $0$-handle, but we note that this was a choice of convenience rather than necessity.  In particular, Lemma 4.5 --- the key technical lemma in the proof of the existence of Morse structures --- explicitly covers the case of multiple index $0$ critical points. 
\end{remark}

\section{Morse structures}

\subsection{Extendable monodromy}

For fixed $S$, there are many possible subsurfaces $P$ so that $(S,P)$ admits a compatible handle structure, and some such subsurfaces  will contain others. If $P \subseteq P'$ and the monodromies $h: P \rightarrow S$, $h' : P' \rightarrow S$ satisfy $h'|_P = h$, then we say $h'$ \emph{extends} $h$ or  that $h$ \emph{extends} to $P'$.

\begin{lemma}
If $h' : P' \rightarrow S$ extends $h : P \rightarrow S$, then there is a contact embedding of $M(S,P,h)$ into $M(S,P',h')$.
\end{lemma}

\begin{proof}
Consider the construction of the contact manifolds via handlebodies $H,N$ and $H',N'$, respectively. The construction of $H$ is independent of $h$ and $P$, so $H,H'$ are contactomorphic. The construction of $N,N'$ shows that $N$ contact embeds in $N'$. Now the gluing of $H$ and $N$ into $M(S,P,h)$, and the gluing of $H'$ and $N'$ into $M(S,P',h')$, respect this contact embedding.
\end{proof}

\begin{definition} 
A monodromy map $h: P \rightarrow S$ is \emph{extendable} if it extends to $S$, i.e., if there exists a homeomorphism $\tilde{h} : S \rightarrow S$ such that $\tilde{h}|_P = h$.
\end{definition}

Thus, when $h$ is extendable, $M(S,P,h)$ contact embeds into $M(S,S,\tilde{h}) = M(S,\tilde{h})$, a closed manifold. This fact  will allow us to use the results of \cite{Gay_Licata15} in the context of partial open books.

In general, a monodromy map for a partial open book is not extendable. For instance, if $h$ is extendable then $S \setminus P \cong S \setminus h(P)$, a condition which often fails; see, for example, Example~\ref{eg:wtf}. However,  certain conditions guarantee that $h$ is  extendable.

\begin{proposition} If $S\setminus P$ and $S\setminus h(P)$ are both connected, then $h$ extends to a homeomorphism of $S$.
\end{proposition}

\begin{proof} Boundary components of $S\setminus P$ and $S\setminus h(P)$ are in bijective correspondence, as $\partial S \setminus \partial P$ is preserved  and arcs of $\partial P \cap \text{Int } S$ map to arcs connnecting the same pairs of points on $\partial S \setminus \partial P=\partial S \setminus \partial h(P)$. Since the Euler characteristic and number of boundary components of these surfaces agree, they are homeomorphic.  A homeomorphism between connected surfaces may be chosen to induce any permutation of the boundary components; this is easily seen by viewing the boundary components as marked points on a closed surface and braiding them.  Thus the map fixing points of $\partial S \setminus \partial P$ may be extended to a homeomorphism of $S$ which sends $P$ to $h(P)$, as desired.
\end{proof}

Figure~\ref{fig:eoex} provides an example of an extendable monodromy.

\subsection{Morse diagrams for extendable partial open books}\label{sec:diagram}

Section~\ref{sec:slice} introduced a slice as a combinatorial encoding of the pair $(S,P)$.   In order to completely encode a partial open book via slices, it remains to encode the map $h: P\rightarrow S$.

We begin by building up Morse functions on $S \times [-1,1]$.

\begin{definition}\label{prop:exist1} 
Given a homeomorphism $\widetilde{h}:S\rightarrow S$ which restricts to the identity on $\partial S$, a  smooth function $F:S \times [-1,1] \rightarrow (-\infty, 0]$  is a \textit{Morse structure function} for $\widetilde{h}$ if  the following properties are satisfied:
\begin{itemize}
\item $F^{-1} (0)= \partial S\times [-1,1]$; 
\item for all values of $t \in [-1,1]$, on the interior of the page $S \times \{t\}$, $F$ restricts to a Morse function $f_t$ with finitely many index $0$ critical points and no index $2$ critical points;
\item $f_t$ is Morse-Smale except  at isolated $t$ values, called handleslide $t$-values; 
\item $f_{-1} \circ h = f_{1}$, where we regard $h$ as a function $S \times \{1\} \rightarrow S \times \{-1\}$
\end{itemize}
\end{definition}

A Morse structure function $F : S \times [-1,1] \rightarrow (-\infty, 0]$ descends to $M(S, \tilde{h})$ and then restricts to a function $M(S,P,h) \rightarrow (-\infty,0]$, also denoted $F$. We call a function of this form a \textit{Morse structure function} for the partial open book.

\begin{definition} A \textit{Morse structure} on $M(S,P,h)$ is a Morse structure function $F$ together with a vector field $V$ such that the following conditions are satisfied:
\begin{enumerate}
\item the handle structures induced by $f_t$ are isotopic for all $t\in [0,1]$;
\item $V$ is tangent to each page;
\item  the restriction of $V$ to the page $X\times \{t\}$ is gradient-like for $f_t$
\item near each  component of the binding, there is a neighbourhood parameterised by $(\rho, \mu, \lambda)$ such that $B=\{\rho=0\}$, $\mu=t$, $F=-\rho^2$, and $V=-(\frac{\rho}{2})\partial_\rho$.  
\end{enumerate}
\end{definition}
Strictly speaking, $f_0$ and $f_1$ are defined on $S \times \{0,1\}$, while $f_t$ is defined only on $P \times \{t\}$ for $t \in (0,1)$. Condition 1 above refers to $f_t|_P$ for $t \in \{0,1\}$.

\begin{proposition}\label{prop:exist} Every partial open book with extendable monodromy admits a Morse structure.
\end{proposition}

\begin{proof}
This is immediate from Proposition 3.3 of \cite{Gay_Licata15}; this is a result about a (non-partial) monodromy map for a standard (non-partial) open book.  It is implicit in the proof there that handleslides can happen at chosen values of $t$; we choose them not to happen for $t \in (0,1)$.
\end{proof}

A Morse structure induces a handle structure on $S \times \{t\}$. In particular, on each page the flowlines between index $0$ and index $1$ critical points, together with the flowlines from the index $1$ critical points to $\partial S \times \{t\}$, form a core complex on  $S \times \{t\}$. Thus $\tilde{h}$ yields a slice $\mathcal{SL}_t$ on $S$ for each value of $t$.

\begin{lemma}\label{lem:det}
The slices on $S \times  \{-1\} $ and $S \times\{0\}$ determine the mapping class of $\widetilde{h}$.  

\end{lemma}

\begin{proof} According to Proposition 2.8 in \cite{Farb_Margalit_MCG}, there is a unique mapping class which renders the core complex of $S\times \{-1\}$  isotopic to that of $S\times \{0\}$.  The lemma then follows from the observation that a slice determines these decorations up to isotopy. As the handle structures are isotopic for $t \in (0,1)$, it is sufficient to look at $t$ from $-1$ to $0$.
\end{proof}


We now consider the slices derived from the \emph{partial} monodromy $h$, taking a Morse structure $(F,V)$ as above. We restrict the slices from $\tilde{h}$ on $S \times [-1,1]$ to $S \times [-1,0] \cup P \times [0,1]$. As $P$ is a collection of handles added to $R$, for each $t \in [0,1]$ we obtain a ``slice" on $P \times \{t\}$, again denoted $\mathcal{SL}_t$, consisting of the oriented arcs and circles of $A = \partial P \cap \partial S$, together with pairs of points from co-cores of primary handles. (There are now no antecedent pairs, nor markers, since these arise from $R$, rather than $P$.)  

Let us now consider all the slices simultaneously. For each $t \in [-1,0]$, we have a slice $\mathcal{SL}_t$ consisting of the oriented $\partial S$ with pairs of antecedent points, primary points, and markers. For each $t \in [0,1]$, we have a slice $\mathcal{SL}_t$ consisting of $A \subseteq \partial S$ with pairs of primary points only.  For any value of $t$, the associated slice embeds as a collection of curves in the corresponding page, and we may assemble these into a  surface embedded in $M(S, P, h)$.

\begin{definition} 
Given an extendable partial open book $(S,P,h)$ and a Morse structure  $(F,V)$, the associated \emph{Morse diagram} is the surface formed from the union of slices
\[
\bigcup_{t \in [-1,1]} \; \mathcal{SL}_t \times \{t\}.
\]
\end{definition}

Thus, the Morse diagram consists of
\[
\partial S \times [-1,0] \cup A \times [0,1]
\]
with the identification $(x,1) \sim (x,-1)$ for all $x \in A$, together with some decorations. (Note the gluing is straightforward since the restriction of $\tilde{h}$ to $\partial S$ is the identity.) The decorations consist of curves, assembled from the points on each slice.
Thus if a slice with $t \in [-1,0]$ has $r$ antecedent pairs, $p$ primary pairs, and $q$ markers, the the Morse diagram contains $r$ pairs of \emph{antecedent curves}, $p$ pairs of \emph{primary curves}, and $q$ \emph{marker curves}. However, the marker curves need not be drawn, as their location is seen automatically seen: markers correspond to arcs of $\partial S \setminus \partial P$, which arise as segments of the boundary of the Morse diagram.  Note that these curves cannot be assumed to be either connected or disjoint from each other; a handle slide of one co-core over another leads creates a \textit{teleport} of the curve associated to the sliding co-core over the curve associated to the stationary co-core; a handleslide on the page $S\times \{t_0\}$ corresponds to a pair of trivalent vertices on the Morse diagram at height $t_0$.  See Figure \ref{fig:exex}.

\begin{figure}
\begin{center}
\includegraphics[scale=0.6]{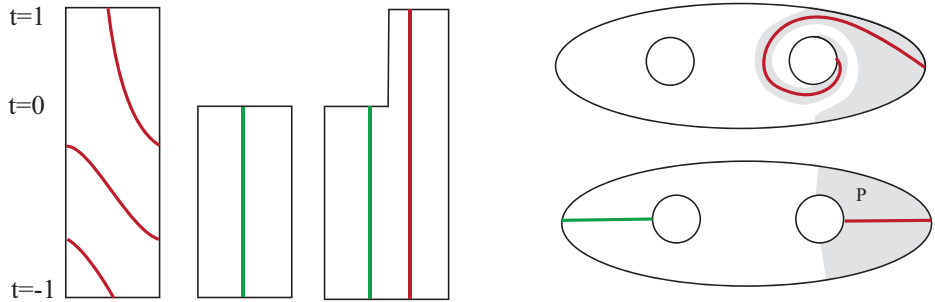}
\caption{Left: A Morse diagram for a partial open book. Right: The monodromy $h$ is defined by its effect on $P$ shown Note that $h$ is extendable to $\widetilde{h}$ which is a single left handed Dehn twist.}
\label{fig:exex}
\end{center}
\end{figure}

Lemma~\ref{lem:det} and the discussion above establish the following result:
\begin{proposition} A Morse diagram determines a partial open book $(S, P, h)$ up to isotopy of the pair $(S,P)$ and the mapping class of an extension $\widetilde{h}\subset \text{MCG}(S)$.
\end{proposition}

\begin{remark} The Morse diagram of a partial open book  will clearly depend on the choice of extension $\widetilde{h}$, but this mirrors the closed case which also makes no claims of uniqueness.
\end{remark}

\section{Front projections of Legendrian tangles}\label{sec:contacto}
If the only goal is constructing a Morse diagram, there is a great deal of flexibility in the choice of $V$. However, strengthening the conditions on $V$ allows us to prove Theorem~\ref{thm:intro} and promotes the Morse diagram to a tool for studying Legendrian tangles in $M(S,P,h)$.  

\begin{proof}[Proof of Theorem~\ref{thm:intro}]

The main result of \cite{Gay_Licata15} is that for each component of the binding $B$ of $M(S,\widetilde{h})$, the preimage of the flow of $V$ is contactomorphic to $(0,\infty) \times S^1 \times S^1$ with coordinates $x \in (0,\infty)$, $y, z \in S^1$ and with contact structure $\xi_W = \ker (dz + x\; dy)$.  

We briefly summarise the idea of the proof and refer the reader to \cite{Gay_Licata15} for details.  The key technical ingredient is a proof that there exists a contact form $\alpha$ and a Morse structure $(F, V)$ with the additional property that $V$ is Liouville for $d(\alpha|_{\text{Int} S\times \{t\}})$.  By choosing $\alpha$ to have a specified form near the binding, we may define an explicit map which sends $(\rho, \mu, \lambda)$ to $\big( \frac{1}{\rho^2},\lambda, \mu)$, where the latter represent $(x,y,z)$ coordinates on $W$.  This map identifies $V$ near the binding with the vector field $x\partial_x$ on $W$  and this identification extends the map to the rest of $M\setminus (\text{Skel}\cup B)$.  
 
 Given this, we consider any extension $\widetilde{h}$ for $h$ and prove Theorem~\ref{thm:intro} by considering the contact submanifold  $M(S,P,h)$ inside $M(S,\widetilde{h})$.  In the case of closed components of the binding of $M(S,P,h)$, the corresponding component of $M\setminus (\text{Skel}\cup B)$ is contactomorphic to $W$ itself, just as in the case of a closed contact manifold.  
 
 For binding components coming from $A$, we begin with a copy of $W$ and remove points which lie in $M(S,\widetilde{h})$ but not $M(S,P,h)$.  The contactomorphism described above takes pages of the open book to planes corresponding to fixed $z$ value.  For simplicity, then, we may assume that $z$ takes values in the circle formed by identifying the endpoints of $[-1,1]$.  For each $z\in [-1,0]$, and annulus $(0,\infty)\times S^1$ is left untouched.  On the other hand, for $z\in (0,1)$, the  circle parameterised by $y$ is identified with a boundary component of $S$; thus when we restrict to the partial open book, we remove $(0,\infty)\times I$ for the image of each interval $I$ in $\partial S\setminus A$.  In the language of flows, we remove the image of any flowline of $V$ which terminates on a point of $\partial S\setminus A$, deleting $|\partial S\setminus A|$ rectangles $J\times (0,1)$ from the Morse diagram. Finally, we note that the complete flowline  from a point on $A$ (away from the co-cores) terminates at an index $0$ critical point.  Since $R$ contains an open neighborhood of each index $0$ critical point, the flowline exits $P$ after some finite amount of time.  Thus for each $y$-interval $K$ which remains, we also remove an open set $\big (0, g(y,z)\big) \times K \times (0,1)$ from $W$; here $g$ is a continuous function.

\end{proof}
Having established (via appeal to the closed case) that one may always find a Morse structure which is compatible with the contact structure as described in the proof of Theorem~\ref{thm:intro}, we henceforth assume all Morse structures are of this form.  Suppose now that $\Lambda$ is a Legendrian curve in $M(S,P, h)$ which is disjoint from the binding and meets the core complex $\mathcal{C}$ transversely.  Viewing the Morse diagram as a properly embedded subsurface of the manifold, we may flow $\Lambda\setminus (\Lambda \cap \mathcal{C})$ by $\pm V$ to the Morse diagram to get a \textit{front}  $\mathcal{F}(\Lambda)$ which is sufficient to recover the original curve.  

\begin{proof}[Proof of theorem \ref{thm:Legendrian_tangles}]
In order to see that the front projection of a Legendrian tangle determines the tangle itself, it is useful to note that $W$ is a quotient of the $x>0$ half-space  of $(\mathbb{R}^3, \xi_{\text{std}})$.  Front projection for Legendrian knots is classically defined in $\mathbb{R}^3$, with the key characteristic that the slope of the tangent in the projection recovers the $x$ coordinate of the Legendrian curve.  Alternatively, one may take the perspective that front projection to the $x=c$ plane in $\mathbb{R}^3$ is the image under the flow of the vector field $x\partial_x$; this vector field is Liouville for the area form induced by $\alpha=dz+x \ dy$ on each plane $z=c$.  The contactomorphism described above takes the Liouville vector field on each page $x\partial_x$ and identifies the image of an $x=c$ plane with the Morse diagram.  The property that a classical front completely determines a Legendrian curve then implies the analogous statement in the context of open books.

\end{proof}

The relationship between fronts in open books and fronts in $\mathbb{R}^3$ yields the familiar properties:
\begin{enumerate}
\item $\mathcal{F}(\Lambda)$ determines $\Lambda$, as the slope of the tangent to  $\mathcal{F}(\Lambda)$ records the flow parameter;
\item $\mathcal{F}(\Lambda)$ is smooth away from finitely many semicubical cusps;
\end{enumerate} 

On the other hand, fronts in partial open books have some new features:
\begin{enumerate}

\item the slope of $\mathcal{F}(\Lambda)$ is negative except where it has an endpoint on the image of $\mathcal{C}$; this follows from the description of $W$ as a quotient of the $\{x<0 \}$ half space in $\mathbb{R}^3$.

\item for $t\in (0,1)$, the slope of $\mathcal{F}(\Lambda)$ is bounded from above by $-\epsilon <0$, as a slope limiting to $0$ corresponds to a Legendrian curve approaching the index $0$ critical point  and $R$ has an open neighbourhood around each index $0$ critical point; 
\item  if $\Lambda$ intersects a core circle $C$ on the $t_0$ page, then $\mathcal{F}(\Lambda)$ will have a pair of \textit{teleporting endpoints} at height $t_0$:  $\mathcal{F}(\Lambda)$ will approach a  curve on the Morse diagram corresponding to $C$ from the left and the other  curve corresponding to $C$ from the right.

\end{enumerate} 
 
 \subsection{Reidemeister moves}
 
 The Reidemeister moves established for Legendrian links in closed contact manifolds extend to a family of moves for fronts of properly embedded Legendrian tangles.

\begin{proof}[Proof of theorem \ref{thm:Legendrian_moves}]
A complete collection of Legendrian Reidemeister moves  for  front projections of Legnedrian knots in open books is given in \cite{Gay_Licata15} and shown  in Figure~\ref{fig:Legendrian_moves} (S, H, K moves).  Since we now consider contact manifolds with convex boundary, we may extend this analysis to properly embedded Legendrian tangles.  The interior of $M(S,P,h)$ is indistinguishable from the interior of a closed contact manifold, so the only  new behaviour on fronts occurs at the boundary of the Morse diagram. Whether one considers these to be new moves is a question of taste; each of the moves listed below is simply the restriction to a Morse diagram for a partial open book of a planar isotopy on a Morse diagram for an ordinary open book.  

\begin{figure}
\begin{center}
\includegraphics[scale=0.7]{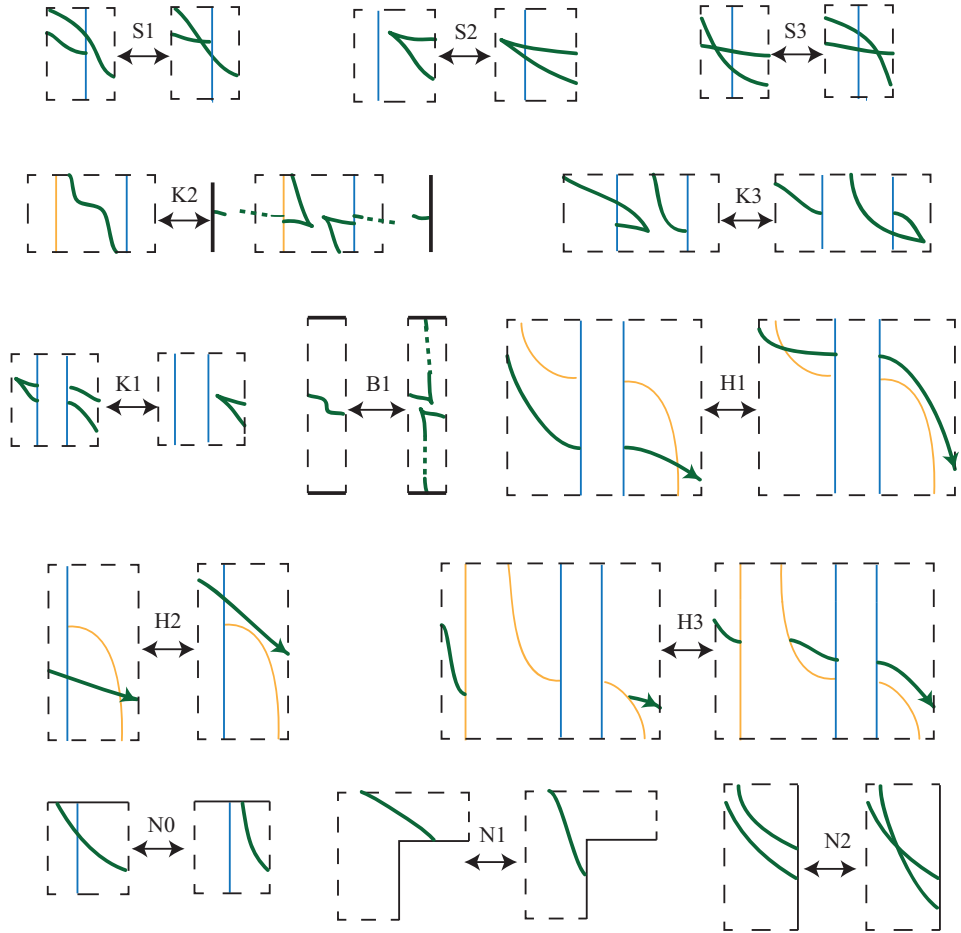}
\caption{Moves for Legendrian links and tangles}
\label{fig:Legendrian_moves}
\end{center}
\end{figure}

   The boundary of the Morse diagram has  three distinct pieces: the \textit{floor}, which is the image under the flow by $V$ of $S\times \{-1\} \setminus h(P)$; the \textit{ceiling}, which is the image of $S\times \{0\} \setminus P$; and the \textit{walls}, which are the image of $\partial P\setminus(\partial S\cap \partial S)\times [0,1]$. In addition to moves on the interior of the diagram which alter the combinatorics of the curves and projection, we see the following moves near the boundary of the Morse diagram:
   
\textbf{Move N0:}  The endpoint of a curve on the front may slide freely along a component of the floor, the ceiling, or a wall, either crossing or teleporting at any trace curve encountered on a floor or ceiling.

\textbf{Move N1:}  The endpoint of a curve on the front may slide right from the ceiling onto a wall and vice versa or left from the floor onto a wall and vice versa. 

\textbf{Move N2:} : Given two curves whose endpoints are near each other on the boundary of the Morse diagram, one may isotope the endpoints past each other, introducing a crossing in the curves.  

 \end{proof}

Move N2 move is reversible, and we note that  it allows $n$ parallel strands with adjacent endpoints may be replaced by the front projection of an arbitrary positive braid.  If performing this isotopy in real time, the slopes at the endpoints must be distinct at the moment of superposition to ensure that the endpoints of the Legendrian curves remain disjoint.

\section{Examples}
\label{sec:examples}

We consider some examples of simple extendable partial open books, Morse structures and front projections.

\begin{example}[Empty monodromy]
Suppose we have a partial open book $(S,P,h)$ where $P$ is empty. Then $h$  is trivially extendable. It is not difficult to see then that $M(S,P,h)$ is just $S \times [-1,1]/\sim$, with dividing set $\partial S \times \{-1/2\}$. Legendrian fronts exist for any Legendrian knots avoiding the skeleton, and as $P$ is empty there is no issue with maximum slope.
\end{example}

\begin{example}[Tight ball]
\label{eg:tight_ball} 
This example also appears in \cite{Etgu_Ozbagci11}. Let $S$ be an annulus and $P$ a thickened properly embedded arc. Let $h$ be a positive Dehn twist, as shown in figure \ref{fig:tight_ball}. A Morse diagram is shown in figure~\ref{fig:tight2}, together with initial and final pages. 
\begin{figure}
\begin{center}
\includegraphics[scale=0.4]{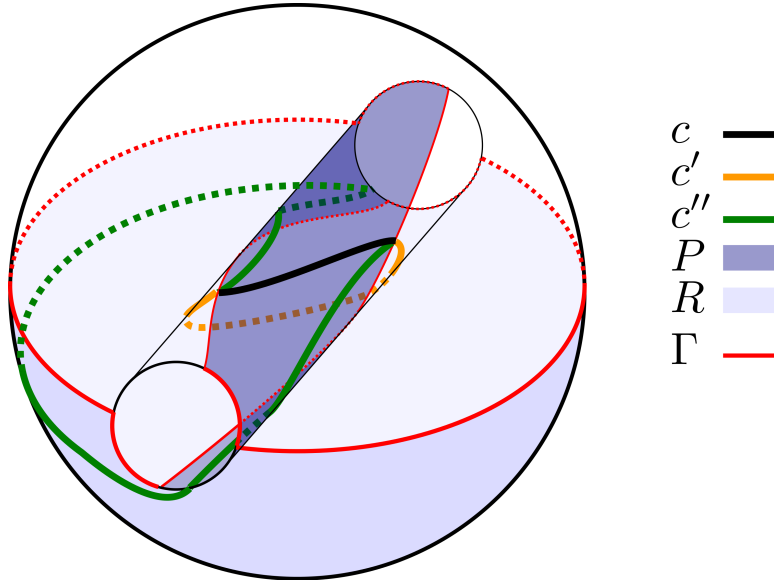}
\caption{The tight ball of example \ref{eg:tight_ball}.}
\label{fig:tight_ball}
\end{center}
\end{figure}

\begin{figure}
\begin{center}
\includegraphics[scale=0.7]{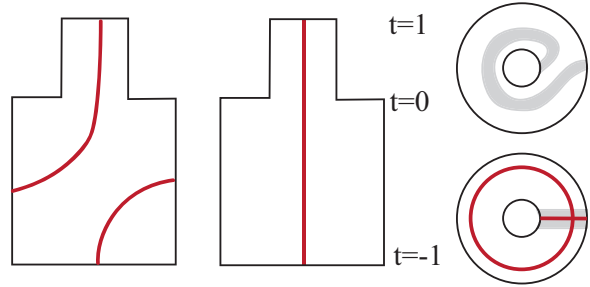}
\caption{A Morse diagram (left) and pages showing initial and final slices (right).  Note that the shaded region in the top right actually the image of $P$ under $h^{-1}$, as required by the identification conventions for the mapping torus.} 
\label{fig:tight2}
\end{center}
\end{figure}

To see why we obtain an tight 3-ball, consider a standard tight contact 3-ball $B$ with connected boundary dividing set $\Gamma$, and positive region $R_+$ a disc. Take a Legendrian arc $\gamma$ properly embedded in $B$, with endpoints on $\Gamma$. Drill out a small tubular neighbourhood $T$ of $\gamma$. Then the dividing set on the resulting surface is shown in figure \ref{fig:tight_ball}. The tube has boundary a cylinder, which is cut into two rectangles by the dividing set. One of these rectangles is $P$. The tube can be regarded as $P \times [0,1]$, and its complement can be regarded as $S \times [-1,0]$ where $S$ is an annulus, consisting of $P$ together with $R=R_+$. A co-core arc $c$ as shown, when pushed across the tube to $c'$, is isotopic in the complement of $T$ to the arc $c''$ on $S$. Then the monodromy takes $c''$ to $c$.
\end{example}

\begin{example}[$S^2 \times I$]
Let $S$ be a disc, $P$ a thickened properly embedded arc. Then $h$ must be isotopic to the identity. So $M(S,P,h)$ consists of a ball $D^2 \times [-1,0]$, with a disc $P \times [0,1]$, glued to a closed curve on its boundary, forming an $S^2 \times I$. This in fact extends to the identity $\tilde{h} : S \rightarrow S$, which produces the tight $S^3$, and hence the contact structure here is the unique tight one. \end{example}

\begin{example}[Overtwisted ball]
\label{eg:overtwisted_ball}
Let $S$ again be an annulus and $P$ a thickened properly embedded arc, as in lower right picture in  figure \ref{fig:tight2}, but now let $h$ be a negative Dehn twist. A Morse diagram is shown in figure~\ref{fig:OTball}, together with a $tb=0$ unknot bounding an overtwisted disc.

\begin{figure}
\begin{center}
\includegraphics[scale=0.4]{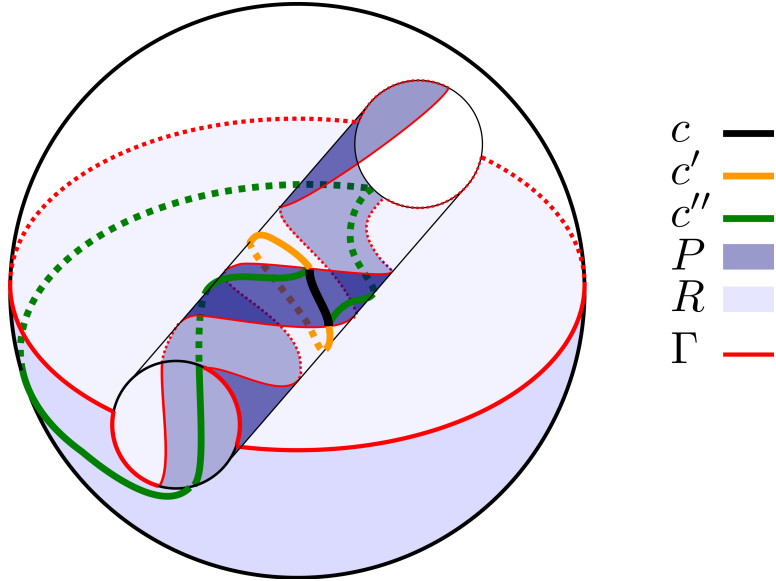}
\caption{The overtwisted ball of example \ref{eg:overtwisted_ball}.}
\label{fig:overtwisted_ball}
\end{center}
\end{figure}

\begin{figure}
\begin{center}
\includegraphics[scale=0.7]{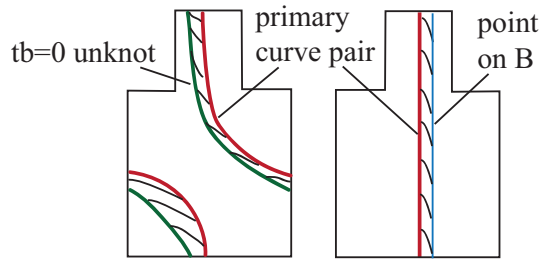}
\caption{An overtwisted disc in  Example~\ref{eg:overtwisted_ball}. The boundary of the disc is parallel to one of the primary curves, and the thinner lines indicate a foliation of the disc by Legendrian curves that teleport across the primary curve and meet at a point on the binding represented by a vertical line. (Example due to Dave Gay.)}
\label{fig:OTball}
\end{center}
\end{figure}

The manifold $M(S,P,h)$ is shown in figure \ref{fig:overtwisted_ball}. As in example \ref{eg:tight_ball}, we drill a tube $T$ out of a ball. However now the dividing set on the tube twists in the opposite direction (the ``wrong way") around the tube. Thus the ball is overtwisted: even if both the tube and its complement are tight, one can find an attaching arc on the tube containing bypasses on both sides. One can again take a co-core curve $c$, trace it through $T$ to $c'$ and through the complement of $T$ to $c'' \subset S$ to show that the monodromy is the restriction of a left-handed Dehn twist.

Indeed, a Legendrian unknot of Thurston-Bennequin number zero can be seen explicitly from its front projection. The leftwards direction of the Dehn twist means that we can draw the front shown in figure~\ref{fig:OTball}.  This unknot  avoids all curves of the Morse diagram and bounds an overtwisted disc that lies in a subset of $W$.  This disc can be seen explicitly on the Morse diagram, as an overtwisted disc admits a radial foliation by Legendrian curves, each of which can also be projected to the diagram.  These curves terminate on a vertical line which represents a single point on $B$.
\end{example}

\begin{example}
\label{eg:wtf}
We conclude with an example which breaks several of the conventions already established, but nevertheless illustrates an interesting phenomenon.  The right hand pictures in Figure \ref{fig:nonexex}  show  initial and final pages specifying a monodromy $h$ which is not extendable.  By way of proof, consider an arc in $S\setminus P$ connecting two distinct boundary components; no arc with the same endpoints exists in $S\setminus h(P)$.
On the other hand, 
this monodromy nonetheless appears to have a perfectly valid Morse diagram, in the sense that the left hand figure defines a sequence of handle slides and isotopies taking the initial core complex to the terminal one. Examples such as these may be interesting for further study.

\begin{figure}[h]
\begin{center}
\scalebox{.6}{\includegraphics{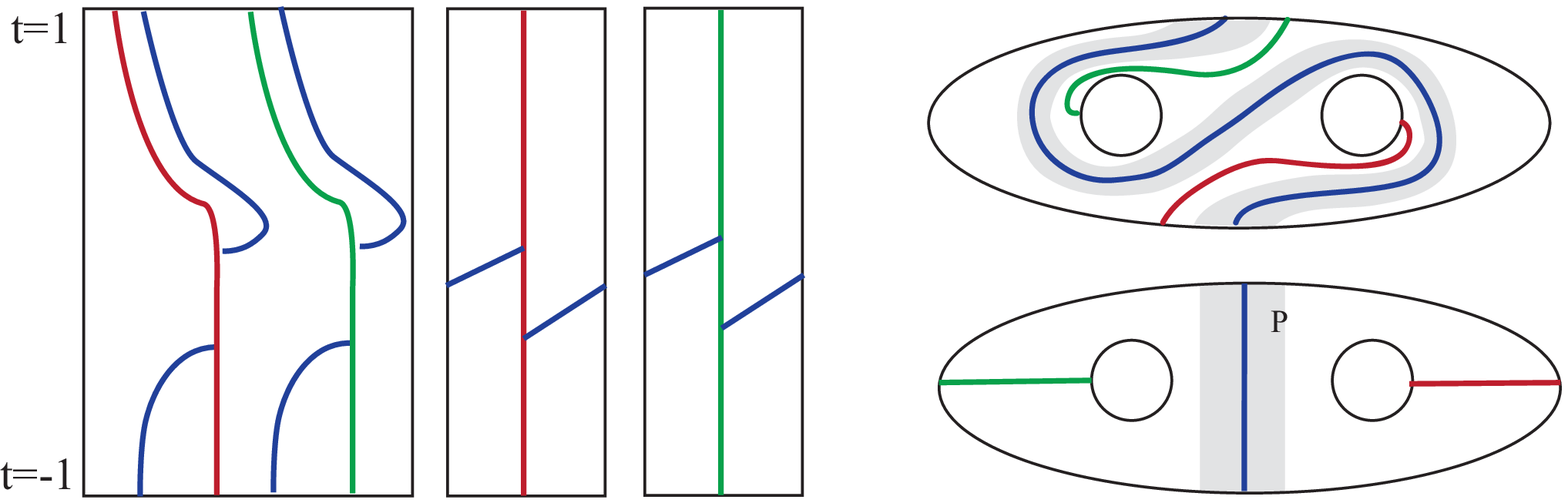}}
\caption{Example~\ref{eg:wtf}
}\label{fig:nonexex}
\end{center}
\end{figure}

\end{example}

\bibliography{biblio}
\bibliographystyle{spmpsci}

\end{document}